\documentclass{amsart}
\usepackage{graphicx}
\usepackage{amssymb}
\usepackage{enumerate}

\usepackage{times}

\newtheorem{theorem}{Theorem}
\newtheorem{lemma}[theorem]{Lemma}
\newtheorem{definition}{Definition}

\def\N{\mathbb{N}}
\def\rarr{\rightarrow}
\def\Sep{\ ;\;}
\def\DC{DC}
\def\Orb{\mathop{\rm Orb}}

\begin{document}
\title{Omega-limit sets and invariant chaos in dimension one}
\author{Michal M\'alek}
\email{Michal.Malek@math.slu.cz}
\address{Mathematical Institute in Opava, Silesian University in Opava, 
  Na Rybn\'{\i}\v{c}ku~1, 746 01 OPAVA, Czech Republic}
\thanks{Preprint version. The final version will be published in Journal of Difference Equations and Applications (doi:10.1080/10236198.2015.1106485)}

\begin{abstract}
Omega-limit sets play an important role in one-dimensional 
dynamics. During last fifty year at least three definitions
of basic set has appeared. Authors often use results with different 
definition. Here we fill in the gap of missing proof of 
equivalency of these definitions.

Using results on basic sets  
we generalize results in paper [P.~Oprocha, 
Invariant scrambled sets and distributional chaos, 
Dyn. Syst. 24 (2009), no. 1, 31--43.] 
to the case continuous maps of finite graphs. 
The Li-Yorke chaos is weaker than positive topological 
entropy. The equivalency arises when we add condition 
of invariance to Li-Yorke scrambled set. 

In this note 
we show that for a continuous graph map properties 
positive topological entropy; horseshoe; invariant 
Li-Yorke scrambled set; uniform invariant distributional 
chaotic scrambled set and distributionaly chaotic pair 
are mutually equivalent.

\end{abstract}
\maketitle

\section{Introduction and main result}

In their famous paper \cite{LY} Li and Yorke defined the notion of scrambled 
set. Their approach of a set containing points which form so-called 
Li-Yorke pair then became one of the most acceptable definitions of 
deterministic chaos. 
In 1986 Sm\'\i tal (\cite{Sm}) proved that Li-Yorke chaos is weaker than 
positive topological entropy, for interval maps. 

Since Li and Yorke paper many other definitions of chaos were offered. One 
of the important generalizations has been provided by Schweizer and 
Sm\'\i tal in \cite{ScSm}. 
Popularity of their distributional chaos comes from the fact that it is 
equivalent to positive topological entropy in one dimensional cases, 
contrary to Li-Yorke chaos.
In this note we focus on chaos of maps of topological graphs.

An {\em arc} is any topological space homeomorphic to the compact
interval $[0,1]$.  A {\em graph} is a continuum (a nonempty compact 
connected metric space) which can be written as the union of finitely 
many arcs any two of which can intersect only in their endpoints 
(i.e., it is a one-dimensional compact connected polyhedron).
We endow a graph $G$ with metric $\varrho$ of the shortest path. In 
such setting we consider a continuous map $f \colon G \rightarrow G$. 

For any $x \in G$, the set of accumulation points of the sequence 
$(f^{n}(x))_{n=1}^{\infty}$ is called the $\omega$-limit set of $x$.
Let $A \subset G$, by $\Orb A$ we understand the set $\{f^{n}(x); 
x \in A \hbox{ and } n\in\N\}$.
%
We say that $f$ has a {\em horseshoe} ($f$ is {\em turbulent}) 
if there are disjoint arcs $U$ and $V$ such that 
$$f(U)\cap f(V)\supset U\cup V.$$

We say that $f$ has a {\em Li-Yorke pair} if there are points $x,y$ such 
that 
$$
  \liminf_{n \rightarrow \infty} \varrho(f^{n}(x), f^{n}(y)) = 0,
  \quad \hbox{and} \quad
  \limsup_{n \rightarrow \infty} \varrho(f^{n}(x), f^{n}(y)) > 0.
$$

Now we proceed with the definition of distributional chaos.
For any two points $x,y\in G$, any positive integer $n$, and any 
real $t$ we define 
$$ 
  \xi (x,y,n,t) 
                = \#\{i\Sep 0 \leq i < n\hbox{ and } \varrho(f^{i}(x),f^{i}(y)) < t \}. 
$$
Now define the {\em upper distributional} and 
{\em lower distributional functions} of the points $x$ and $y$ by 
formulas 
$$ 
  F^{xy} (t) = \limsup_{n\rightarrow\infty} \frac{1}{n} \xi (x,y,n,t), 
$$ 
$$ 
  F_{xy} (t) = \liminf_{n\rightarrow\infty} \frac{1}{n} \xi (x,y,n,t) 
$$ 
respectively. 
The points $x$ and $y$ form a {\em \DC1-pair} for $f$ if 
$$
  F^{xy}\equiv \chi_{(0,\infty)} \ {\rm and}  \  F_{xy}(\varepsilon)=0, 
    \ {\rm for \ some}  \ \varepsilon >0,
$$
they form a {\em \DC2-pair} if 
$$
  F^{xy}\equiv \chi_{(0,\infty)} \ {\rm and}  \  F_{xy}(0+) < 1, 
$$
and finally they form a {\em \DC3-pair} if 
$$
  F^{xy}(t) > F_{xy}(t),  
    \ {\rm for \ any} \ t \ {\rm in \ an \ interval}.
$$
Obviously every \DC1-pair is a \DC2-pair and every \DC2-pair is a \DC3-pair.

%
The notion of distributional chaos was introduced in \cite{ScSm}, but without 
naming versions explicitly. In the same paper there was proved that for interval 
maps all three versions are mutually equivalent and also equivalent to 
positive topological entropy. 
These three versions was explicitly named in \cite{S2, BS2} and 
discussed their properties in general spaces and particularly in the case of 
skew-product maps.
The relations between the versions of distributional chaos as well as the 
relations to other important notions as topological entropy, horseshoes and 
other types of chaos was studied in the case of circle, graphs and dendrites 
(see \cite{HM,MO}). 
In the case of graph maps all three types are mutually equivalent and therefore 
we can formulate the last condition of our main theorem for any type of \DC-pair.

Let $S \subset G$ be a set containing more than one point. 
We say that $S$ is a {\em Li-Yorke scrambled set} if any two different points 
from $S$ form a Li-Yorke pair. Similarly we define {\em \DC1 (\DC2 and \DC3) 
scrambled set}. 
There are several ways defining $f$ to be chaotic (in Li and Yorke 
or distributional case) based on the size of its scrambled set --- two-point, 
infinite or uncountable. 
We refer to \cite{GL} for further reading on the size of scrambled sets.
In the case of continuous graph maps the existence of a two-point distributional 
scrambled set is equivalent to the existence of an uncountable one, similarly 
for Li-Yorke scrambled sets (see \cite{K} for details).

If a scrambled set $S$ satisfies condition $f(S) \subset S$ 
we call it {\em invariant.} 
As it was mentioned, the existence of a Li-Yorke scrambled set is weaker than
positive topological entropy. This is no longer true when we add the assumption 
of invariance of a scrambled set. This was proved in \cite{O} for interval maps 
and here we generalize the results for the case of graph maps.

A \DC1 scrambled set $S$ is called {\em uniform} if there is an $\varepsilon > 0$ such 
that $F_{xy}(\varepsilon) = 0$, for any pair of different points $x$ and $y$ 
from $S$.

We recall the definition of the topological entropy by Bowen (\cite{Bo}) and
Dinaburg (\cite{D}). Let $\varepsilon > 0$ and $n$ be a positive integer. 
A set $A \subset G$ is an \textit{$(n, \varepsilon)$-separated set} if for each
$x,y \in A$, $x \neq y$, there is an integer $i$, $0 \leq i < n$, such that 
$\varrho (f^{i}(x),f^{i}(y)) > \varepsilon$. 
Let $s_{n}(\varepsilon)$ denote the maximal cardinality of all $(n, \varepsilon)$-separated sets. 
The {\em topological entropy of $f$} is 
\begin{equation}
  h(f) = \lim_{\varepsilon \rightarrow 0} 
         \limsup_{n \rightarrow \infty} \frac{1}{n} \log s_{n}(\varepsilon).\label{eq-entropy} 
\end{equation}

Let $f\colon G \rightarrow G$ and $g\colon K \rightarrow K$ be continuous 
maps of graphs. A continuous surjection $\varphi\colon G \rightarrow K$ 
which is monotone (i.e. $\varphi^{-1}(y)$ is connected for all $y \in K$) 
is called {\em semiconjugation} if $\varphi \circ f = g \circ \varphi$. 
Moreover if there exists an $k \in \N$ such that 
$\#\varphi^{-1}(x) \leq k$ for each $y \in K$ then we say that $\varphi$ 
semiconjugates $f$ and $g$ {\em almost exactly}.

%
\section{Basic sets}
In the sixties, A.~N.~Sharkovsky has systematically studied properties of 
the $\omega$-limit sets of the  continuous maps of the interval (cf., e.g.,  
\cite{Sh2} and \cite{Sh3}). He shown that any such  set is contained 
in a unique maximal $\omega$-limit set and introduced three types of maximal 
$\omega$-limit set: cycle, first type (lately known as solenoids) and second 
type (lately known as basic sets). 

\begin{definition}[S-basic set]
A maximal $\omega$-limit set $\omega$ for continuous map of compact interval
is called a basic set if $\omega$ is infinite and it contains a periodic point. 
\end{definition}
Although Sharkovsky's definition was originally formulated for the case of 
continuous maps of compact interval, it can be used in the same form  
for graph maps too.

Lately in the eighties A. Blokh widely extend properties of basic sets 
but using different definition (cf. e.g. \cite{Bl}). 
\begin{definition}[B-basic set]
Let I be an $n$-periodic interval, $\Orb I = M$. Consider a set $B(M,f) 
= \{ x\in M: $ for any relative neighborhood $U$ of $x$ in $M$ we have 
$\overline{\Orb U} = M \}$. Set $B(M,f)$ is called a basic set if it is 
infinite.  
\end{definition}
This definition is used, without explicit formulation, in papers \cite{Bl123} 
for the case of graph maps by taking periodic closed connected subgraph 
instead of $I$.

In paper \cite{HM} was Sharkovsky's classification of maximal $\omega$-limit sets 
extended to the case of graph maps. It was necessary to add fourth type 
of maximal $\omega$-limit sets to cover graph specific cases of $\omega$-limit
sets (singular sets). 
Let ${\omega}$ be a maximal $\omega$-limit set of a continuous map $f$ 
of graph. Put
$$
P_{{\omega}}=\bigcap_{U}\overline{\Orb U}
$$
where $U$ is taken over all neighborhood intersecting $\omega$.  
Note that $P_{\omega}$ is closed and strongly invariant i.e. 
$f(P_{\omega}) = P_{\omega}$.
Using cardinality of $P_{\omega}$ and relation $\omega$ to the set 
of periodic points we can divide the set of all maximal 
$\omega$-limit sets into four classes.

If $\omega$ is finite then it is a {\em cycle}. Now consider $\omega$ to be infinite.
If $P_{\omega}$ is a nowhere dense set then we call $\omega$ a {\em solenoid}.  
If $P_{\omega}$ consists of finitely many connected
components and $\omega$ contains no periodic point then we call
$\omega$ a {\em singular set}.  
\begin{definition}[HM-basic set]
If $P_{\omega}$ consists of finitely many connected components and
$\omega$ contains a periodic point then $\omega$ is called a {\em basic
set}.  
\end{definition}

Use of results on basic sets under different definitions become folklore nowadays. 
Here we offer proof of equivalency of mentioned definition of basic sets.   

\begin{theorem}
Definitions of S-basic set, B-basic set and HM-basic set are mutually equivalent. 
\end{theorem}
\begin{proof}
  Obviously every HM-basic set is S-basic set. Conversely, let $\omega$ is 
S-basic set, then by must be one of class (in sense of \cite{HM}) i.e. cycle, 
solenoid, singular set and HM-basic set. Since $\omega$ is infinite it 
cannot be cycle. It also cannot be singular set because $\omega$ contains 
a cycle. And if $U$ is a neighborhood of a point of $\omega$ then by
\cite[Theorem 3.7]{ScSm} $\overline{\Orb U}$ covers $\omega$. Therefore 
$\omega$ cannot be solenoid and $\omega$ must be HM-basic set. 

	Let $\omega$ is a HM-basic set. Then $P_{\omega}$ is finite union 
of periodic subgraphs and we put $M=P_{\omega}$. For each $x \in \omega$ 
and its neighborhood $U$ we have $\overline{\Orb U} \supset M$. This 
follows that $\omega$ is B-basic set. 
Now let $\omega$ be $B$-basic set. Again we use fact that $\omega$ must be one 
of mentioned class. It cannot be a cycle since it is infinite and 
singular set as well because it contains periodic point (see \cite[Theroem 4.1]{Bl123}). 
Finally by \cite{Bl123} $\omega$ is maximal and $f$ has positive topological entropy on $\omega$ 
which is impossible on solenoids. Therefore $\omega$ is a HM-basic set.
\end{proof}

Having equivalency of mentioned definition we can formulate the following 
theorem based on results from \cite{Bl}, \cite{Bl123} and \cite{HM}.
\begin{theorem}\label{PropBasSets}
Let $f$ be a continuous graph map and $\omega$ its basic set then 
\begin{enumerate}[(i)]
  \item $\omega$ is a perfect set;
  \item the system of all basic sets of f is countable;
  \item $h(f) \geq (\log 2)/(2n)$, where $n$ is number of periodic portion of $\omega$;
  \item periodic points are dense in $\omega$;
  \item $f|_\omega$ is semiconjugated with $g$ almost exactly, where $g$ is a continuous 
        transitive map of a graph;
  \item $f^{n}$ has a horseshoe, for some $n$;
  \item let $P_\omega$contains no proper periodic subgraph,  
        $U \subset \mathop{\rm int} P_\omega$ and $J \cap \omega$ be infinite then 
        $f^{n}(J) \supset U$ for sufficiently large $n$.
\end{enumerate} 
\end{theorem}

\section{Invariant chaos}

The following theorem is the main result of this note.
\begin{theorem}\label{mainThm}
Let $f$ be a continuous graph map then the following 
conditions are equivalent: 
\begin{enumerate}[(i)]
\item topological entropy of $f$ is positive;
\item there exists an $n \in \N$ such that $f^{n}$ has a 
  horseshoe (is turbulent);
\item there exists an $n \in \N$ such that $f^{n}$ has an 
  invariant Li-Yorke scrambled set;
\item there exists an $n \in \N$ such that $f^{n}$ exhibits 
  uniform distributional chaos with an invariant \DC1-scrambled set;
\item $f$ has a \DC\ pair of any type.
\end{enumerate}
\end{theorem}

The following lemma is known for interval maps, we generalize it 
for the case of graph maps. 
\begin{lemma}
Let $f$ be a continuous graph map. Topological entropy of $f$ is 
positive if and only if there is an $\omega$-limit set  
containing a cycle but different from this cycle. 
\end{lemma}
\begin{proof}
Suppose that there is an $\omega$-limit set $\omega$ containing 
a cycle but different from this cycle. By \cite{MS} 
$\omega$ is contained in a maximal $\omega$-limit set $\tilde{\omega}$.
Set $\tilde{\omega}$ must be infinite (otherwise it would be a cycle) 
and therefore $\tilde{\omega}$ is a basic set. By Theorem~13 in 
\cite{HM} $f$ has positive entropy.

The converse implication easily follows from Theorem~13 in \cite{HM}
\end{proof}

\begin{lemma}\label{lemmaInvar}
Let $f$ be a continuous graph map. 
If $f$ has an invariant Li-Yorke scrambled set then $f$ 
has positive topological entropy.
\end{lemma}
\begin{proof}
Let $x$ be a point of invariant Li-Yorke scrambled set for $f$. 
The condition 
$$
  \liminf_{i\rarr\infty} \varrho(f^{i}(x),f^{i+1}(x)) = 0
$$
implies that $\omega_f(x)$ contains a fixed point. 
On the other hand from the condition 
$$
  \limsup_{i\rarr\infty} \varrho(f^{i}(x),f^{i+1}(x)) > 0
$$Ê
we get $\#\omega_f(x) >Ê1$. We finish the proof by applying 
the previous lemma.
\end{proof}

\begin{proof}[Proof 
of Theorem \ref{mainThm}]
The equivalency of (i) and (ii) has been proved in \cite{LM}. 

When $f^{n}$ has a horseshoe then on a closed invariant subset 
$f^{n}$ is conjugate to the full one-sided shift of two symbols. 
By Theorem~1 in \cite{O} this shift has a uniform invariant 
\DC1-scrambled set. By conjugacy we get scrambled set 
for $f^{n}$ with required properties. This proves (ii) implies (iv).


The implication from (iv) to (i) follows from Theorem~13 in \cite{HM}.

Now apply Lemma~\ref{lemmaInvar} to $f^{n}$ we have that  (iii) 
implies that $h(f^{n}) > 0$. Then by formula $h(f^{n}) = n h(f)$ we 
get (i).

The implication from (iv) to (iii) is trivial. 

The equivalency of (v) and (i) has been proved in \cite{MO}.
\end{proof}

\section*{Acknowledgments}
The author thanks to professors  \v{L}ubom\'{\i}r Snoha and Roman Hric for 
stimulating discussions and useful suggestions.
Research was funded by institutional support for the development of research
organization (I\v{C}47813059).

\end{document}